\documentclass[12pt]{amsart}

\usepackage{amsmath,amsthm,amscd,amsfonts,amssymb,epic,eepic,bbm}
\usepackage[pagebackref,colorlinks=true,linkcolor=blue,citecolor=blue]{hyperref}

\allowdisplaybreaks
\newtheorem{thm}{Theorem}[section]
\newtheorem{cor}[thm]{Corollary}

\newtheorem{lem}[thm]{Lemma}

\newtheorem{prop}[thm]{Proposition}
\theoremstyle{remark}
\newtheorem*{rem}{Remark}

\newcounter{remarkscounter}
\newenvironment{remarks}
{\medskip\noindent{\it
Remarks.}\begin{list}{{\rm(\arabic{remarkscounter})}
}{\usecounter{remarkscounter}

\setlength{\labelsep}{\fill} \setlength{\leftmargin}{0pt}
\setlength{\itemindent}{\fill}
\setlength{\labelwidth}{\fill}\setlength{\topsep}{0pt}
\setlength{\listparindent}{0pt}}} {\end{list}}

\numberwithin{equation}{section}
\newcommand{\A}{\mathbb{A}}
\newcommand{\GL}{\mathrm{GL}}
\newcommand{\SL}{\mathrm{SL}}
\newcommand{\ZZ}{\mathbb{Z}}

\newcommand{\QQ}{\mathbb{Q}}

\newcommand{\lto}{\longrightarrow}

\newcommand{\OO}{\mathcal{O}}
\newcommand{\CC}{\mathbb{C}}
\newcommand{\RR}{\mathbb{R}}

\newcommand{\gl}{\mathfrak{gl}}
\newcommand{\quash}[1]{}

\theoremstyle{definition}

\renewcommand{\bar}{\overline}

\numberwithin{equation}{subsection}

\newcommand{\one}{\mathbbm{1}}

\begin{document}
\title[Invariant automorphic kernels]{Invariant four-variable automorphic kernel functions}
\author{Jayce R. Getz}
\address{Department of Mathematics\\
Duke University\\
Durham, NC 27708}
\email{jgetz@math.duke.edu}


\thanks{The author is thankful for partial support provided by NSF grant DMS-1405708.  Any opinions, findings, and conclusions or recommendations expressed in this material are those of the author and do not necessarily reflect the views of the National Science Foundation.}

\maketitle

\begin{abstract}

Let $F$ be a number field, let $\mathbb{A}_F$ be its ring of adeles, and let $g_{\ell 1},g_{\ell 2},g_{r1}, g_{r2} \in \mathrm{GL}_2(\mathbb{A}_F)$.  
Previously the author provided an absolutely convergent geometric expression for
the four variable kernel function
\begin{align*}
\sum_{\pi} K_{\pi}(g_{\ell 1},g_{r1})K_{\pi^{\vee}}(g_{\ell 2},g_{r2})\mathrm{Res}_{s=1}L(s,(\pi \times \pi^{\vee})^S),
\end{align*}
where the sum is over isomorphism classes of cuspidal automorphic representations $\pi$ of $\mathrm{GL}_2(\mathbb{A}_F)$.  Here 
$K_{\pi}$ is the typical kernel function representing the action of a test function on the space of the cuspidal automorphic representation $\pi$.  In this paper we 
show how to use ideas from the circle method to provide an alternate expansion for the four-variable kernel function that is visibly invariant under the natural action of $\GL_2(F) \times \GL_2(F)$.
\end{abstract}

\tableofcontents

\section{Introduction}

Let $F$ be a number field
 and let $A \leq \GL_2(F_\infty)$ be the central diagonal copy of $\RR_{>0}$.  For $f \in C_c^\infty(A \backslash \GL_2(\A_F))$ and cuspidal automorphic representations $\pi$ of $A \backslash \GL_2(\A_F)$ let
$$
K_{\pi(f)}(x,y)
$$
denote the usual kernel function (for more details on our notational conventions see the introduction to \cite{GetzKer1}).  

Let
$g_\ell=(g_{\ell 1},g_{\ell 2}),g_r=(g_{r1},g_{r2})\in \GL_2(\A_F) \times \GL_2(\A_F)$
and let $f_1,f_2 \in C_c^\infty(A \backslash \GL_2(\A_F))$ be test functions unramified outside of a finite set of places $S$. Let
$$
\Sigma_{\mathrm{cusp}}(g_\ell,g_r):=\sum_{\pi} K_{\pi(f_1)}(g_{\ell 1},g_{r 1})K_{\pi^{\vee}(f_2)}(g_{\ell 2},g_{r2})\mathrm{Res}_{s=1}L(s,(\pi \times \pi^{\vee})^S).
$$
In \cite{GetzKer1} the author gave a geometric expression
for $\Sigma_{\mathrm{cusp}}(g_\ell,g_r)$.  The motivation, as explained in loc.~cit., is to integrate this expression over a pair of twisted diagonal subgroups and thereby provide an explicit nonabelian trace formula, that is, a trace formula whose spectral side
is a weighted sum over representations invariant under a simple nonabelian subgroup of $\mathrm{Aut}_F(\QQ)$.  This is a step in the author's program to establish nonsolvable base change for $\GL_2$ (see \cite{GetzApp}, \cite{GH}, \cite{GetzKer1}).  Other possible applications are given in \S \ref{ssec-ap} below.

The defect in the formula for $\Sigma_{\mathrm{cusp}}(g_\ell,g_r)$ given in \cite{GetzKer1} is that it is not obviously invariant under $(g_\ell,g_r) \longmapsto (\gamma_\ell g_\ell,\gamma_rg_r)$ for $\gamma_\ell,\gamma_r \in \GL_2(F)^{\times 2}$; briefly, it is not invariant under $\GL_2(F)^{\times 2}$.  Thus to use it for its intended purpose one seems to be forced to employ some variant of the Rankin-Selberg method.  

 The root of the lack of invariance in the formula for $\Sigma_{\mathrm{cusp}}(g_\ell,g_r)$ in \cite{GetzKer1} is easy to describe.  In loc.~cit.~ one investigates  
a certain limit constructed out of Whittaker coefficients of a product of kernel functions.  Taking the Whittaker coefficients introduces integrals over adelic quotients of nilpotent groups and destroys invariance, and it is not so easy to rebuild this invariance on the geometric side of the formula.

In this paper we overcome this difficulty by providing a different geometric formula for $\Sigma_{\mathrm{cusp}}(g_\ell,g_r)$ 
that is clearly invariant under $\GL_2(F)^{\times 2}$.  This will make it easier to integrate the formula over a pair of twisted diagonals.  It should also be noted that the approach exposed in this paper should work if we replace $\GL_2$ by an inner form, whereas the approach of \cite{GetzKer1} can't be applied to these groups because it involves integration over nilpotent subgroups.  

Moreover, the approach given here is of interest in itself for at least two reasons.  First, it is an instance where one can insert the nonstandard test functions of Ng\^o \cite{Ngo} and Sakellaridis  \cite{Sak} into the trace formula and understand the coarse analytic properties of the result without appealing to known results on automorphic forms.  Second, it involves a variant of the circle method in a case where one is not interested in the main term, but in the secondary terms (this is discussed in \S \ref{ssec-the-method} below).  Finally, we remark that at this stage in the mathematical community's investigation of Langlands functoriality beyond endoscopy, it is vital to develop as many tools and methods as possible in order to broaden our collective understanding.   

\begin{rem} The approach of \cite{GetzKer1} is not without merits.  It is a little simpler than the approach exposed here in some respects, and it is unclear which method will generalize easiest to the higher rank case.  
\end{rem}

\subsection{Statement of the formula} \label{ssec-formula}
Let $S$ be a finite set of places of $F$ including the infinite and dyadic places.
We assume that $\OO_F^S$ has class number $1$ and $\OO_F/\ZZ$ is unramified outside of $S$.  Let $v \in S-\infty$.  
Let $k \in \ZZ_{ > 0}$.  Consider the following assumption on $\Phi_v \in C_c^\infty(\GL_2(F_v))$:
\begin{itemize}
\item[(A)] The function $\Phi_v$ is supported in the set of $g$ with valuation $v(\det g)=k$,\\ $\int_{\GL_2(F_v)}\Phi_v(g)dg=0$, and $\Phi_v \in C_c^\infty(\GL_2(F_v)//\GL_2(\OO_{F_v}))$.
\end{itemize}
We introduce our test functions and assumptions, and then comment on them after the statement of our main theorem:
\begin{enumerate}
\item[(i)]
Let $f_1,f_2 \in C_c^\infty(A \backslash \GL_2(F_S))$ and define $f(g_1,g_2)=f_1(g_1)f_2(g_2)$ for $(g_1,g_2) \in \GL_2(\A_F)^{\times 2}$.\smallskip
\item[(ii)] Assume that $f_i=f_i^v \otimes f_{iv}$ with $f^v_1,f^v_2 \in C_c^\infty(A \backslash \GL_2(F_{S-v}))$, $f_{1v} \in C_c^\infty(\GL_2(F_v))$, $f_{2v}=\one_{\GL_2(\OO_{F_v})}$, and $f_{1v}$ satisfying assumption (A) above. 
\smallskip
\item[(iii)] Assume that the operators 
$$
R(f_1),R(f_2):L^2(A \GL_2(F) \backslash \GL_2(\A_F))\lto L^2(A \GL_2(F) \backslash \GL_2(\A_F))
$$ induced by the right regular action and the test functions $f_1$ and $f_2$ respectively have cuspidal image.
\smallskip
\item[(iv)] 
Let $V_1 \in C_c^\infty((0,\infty))$ and
 define $V :F_S^{\oplus 2} \to \RR_{ \geq 0}$ as in \S \ref{sec-delta-app} using it.
 \smallskip
\item[(v)] Let $h(x,y):=W(x)-W(y/x)$ where $W=W_\infty\one_{\OO_{FS}^\times} \in C_c^\infty(F_S^\times)$ and $\widehat{W}(0)=1$.  
\smallskip
\item[(vi)] Let $\one_{\mathcal{F}S}$ be the characteristic function of a fundamental domain for $\OO_F^{S \times}$ acting on $F_S^\times$ and for $b=(b_1,b_2) \in F_S^\times \times F_S^\times$ define 
$$
\one_\mathcal{F}(b):=\one_{\mathcal{F}S}(b_1)\one_{\widehat{\OO}_F^{S \times}}(b_1)\one_{\widehat{\OO}_F^{S \times}}(b_2).
$$

\end{enumerate}

We also abbreviate 
\begin{align*}
b\det T:&=(b_1\det T_1,b_2\det T_2)\\
P(b,T):&=b_1\det T_1-b_2 \det T_2\\
\mathrm{tr}\,\gamma T:&=\mathrm{tr}(\gamma_1T_1+\gamma_2T_2)
\end{align*} 
for $b=(b_1,b_2) \in (\A_F^{\times})^{\oplus 2}$, $\gamma=(\gamma_1,\gamma_2)$, and $T=(T_1,T_2) \in \gl_2(\A_F)$.

Finally, define
\begin{align} \label{I-def}
\mathcal{I}(b,\gamma):=\one_{\mathcal{F}}(b)\int_{F_S}\Bigg(\int_{\gl_2(F_S)^{\oplus 2}}\frac{V(b\det T)}{|b_1\det T_1|_S} h\left(t,
P(b, T) \right)  f(T) \psi_S\left(\frac{\mathrm{tr}\,\gamma T}{t}\right)dT\Bigg)\frac{dt}{|t|_S^4}.
\end{align}

We show in Proposition \ref{prop-arch} and Corollary \ref{cor-h} below that $\mathcal{I}(b,\gamma)$ is Schwartz as a function of $\gamma \in \gl_2^{\oplus 2}(F_S)$.

With the notation above in mind, 
we state the main theorem:
\begin{thm} \label{main-thm} 
Letting $\widetilde{V}_1$ denote the Mellin transform of $V_1$, one has that $\Sigma_{\mathrm{cusp}}(g_\ell,g_r)$ is equal to 
\begin{align*} 
\frac{\zeta_E^S(2)|\det g_\ell g_r^{-1}|^2}{d_F^4\widetilde{V}_1(1)}
\sum_{0\neq \gamma\in \gl_2(F)^{\oplus 2}}&\sum_{\substack{b \in (F^\times)^{\oplus 2}\\ b_2\det \gamma_1=b_1\det \gamma_2}}
\sum_c |c|_S^2\mathcal{I}(b\det g_\ell g_r^{-1},cg_r^{-1}\gamma g_\ell )\one_{\gl_2(\widehat{\OO}_{F}^S)}(g^{-1}_r \gamma g_\ell). \nonumber
\end{align*}
where $d_F \in \ZZ_{>0}$ is the absolute discriminant of $F$, the sum on $c$ is over a set of representatives for the nonzero principal ideals of $\OO_F^S$, and $|\det g|:=|\det g_1|| \det g_2|$. 
\end{thm}
Here $d_F \in \ZZ_{>0}$ is the absolute discriminant of $F$.  
We complete the proof of Theorem \ref{main-thm} in \S \ref{sec-ps-in-d} below.  We now comment on the various assumptions:

\begin{remarks}
\item Assumption (A) eliminates the contribution of the nongeneric spectrum; it is no loss of generality for studying the generic spectrum, as we prove in Lemma \ref{lem-no-loss}.

\item We assume (iii) only to simplify the spectral side of the formula; it is not used in the analysis of the geometric side which forms the bulk of this paper.  

\item The $V$ function smooths sums over Hecke operators. 
\item The $h$ function comes up in our application of the $\delta$-symbol method (see \S \ref{sec-delta} and \S \ref{sec-delta-app}).
\end{remarks}

\subsection{Possible applications} \label{ssec-ap}

Our primary motivation for proving Theorem \ref{main-thm} is to use it to produce a trace formula isolating representations invariant under a pair of automorphisms $\iota,\tau \in \mathrm{Aut}_{\QQ}(F)$.  One might then hope to compare this formula with a similar formula over the fixed field of $\langle \iota,\tau \rangle$ acting on $F$ and prove nonsolvable base change for $\GL_2$ (compare \cite{GetzApp}, \cite{GH}, \cite{GetzKer1}).  Of course, this is very speculative.  

Let $\chi_1,\chi_2,\chi_3,\chi_4:F^\times \backslash \A_F^\times \to \CC^\times$ be a quadruple of characters.  A more immediate application of Theorem \ref{main-thm} might be studying asymptotics of sums of products of $L$-functions of the form
$$
L(\tfrac{1}{2},\pi \otimes \chi_1)L(\tfrac{1}{2},\pi \otimes \chi_2)\bar{L(\tfrac{1}{2},\pi \otimes \chi_3)}\bar{L(\tfrac{1}{2},\pi \otimes \chi_4)}
$$
as the analytic conductor of $\pi$ increases.  W.~Zhang has also pointed out to the author the possibility of using the main theorem to prove a new Waldspurger type formula (compare \cite[\S 4.2]{WZ}) involving products of $L$-functions as above.  In any case, we would like to emphasize that Theorem \ref{main-thm} is flexible enough to lead to a variety of applications beyond the primary one motivating the author.

\subsection{The method} \label{ssec-the-method}

Let $f \in C_c^\infty((A \backslash \GL_2(\A_F))^{\times 2})$ and let
$$
\one_{m}:=\one_{g \in \gl_2(\widehat{\OO}_F^S) \cap \GL_2(F):\,\det g \widehat{\OO}_F^{S\times}=m\widehat{\OO}_F^{S\times}}
$$
and for $g=(g_1,g_2) \in \GL_2(\A_F^S)^{\times 2}$ let $\one_{m}(g):=\one_m(g_1)\one_m(g_2)$.
We consider
\begin{align*}
\Sigma(X)=\sum_{m} \frac{V_1(|m|_S/X)}{|m|_SX}
\sum_{\gamma \in \GL_2(F)^{\times 2}}f\one_m(g_\ell^{-1} \gamma g_r)
\end{align*}
where the sum on $m$ is over a set of representatives for the nonzero (principal) ideals of $\OO_F^S$. 

The following proposition is proven using an easy modification of the proof of \cite[Proposition 5.1] {GetzKer1}:
\begin{prop} \label{prop-spec}
If $f=f_1f_2$ and $R(f_1)$, $R(f_2)$ have cuspidal image, then 
$$
\widetilde{V}_1(1)\Sigma_{\mathrm{cusp}}(g_\ell,g_r)=\zeta_F^S(2)\lim_{X \to \infty}\Sigma(X).
$$ \qed
\end{prop} 
\noindent This is the only place in the paper where we use the assumption that $R(f_1)$ and $R(f_2)$ have cuspidal image.  

The bulk of the paper is devoted to evaluating $\lim_{X \to \infty}\Sigma(X)$ geometrically.  
To see what is going on, it is perhaps useful to specialize to the case where $S=\infty$, $F=\QQ$, $g_\ell=g_r=(I,I)$, and where $f \in C_c^\infty((A \backslash \GL_2(\RR))^{\times 2})$ is supported on elements of $\GL_2(\RR)$ with positive determinant.  In this case the sum $\Sigma(X)$ reduces to 
$$
\sum_{\substack{\gamma_1,\gamma_2 \in \gl_2(\ZZ)\\
\det \gamma_1=\det \gamma_2}} \frac{V_1\left(\frac{\det \gamma_1}{X} \right)}{X\det \gamma_1} f(\gamma_1,\gamma_2).
$$
Thus $\Sigma(X)$ is essentially a smoothed version of the function counting integral points of height at most $X$ on the hypersurface in $\gl^{\oplus 2}_2 \cong \A^8$ defined by $\det\gamma_1-\det \gamma_2=0$.  However, we are not interested in the main term, which comes from the trivial representation\footnote{It would have size $X$ if we had not used assumption (ii) of Theorem \ref{main-thm} to remove it (compare Lemma \ref{lem-vanish}).}.  We are interested in all of the secondary terms.   Despite this, the version of the circle method known as the $\delta$-symbol method is still strong enough to give us what we need; a suitable modification of this method is what we use.

\begin{remarks}
\item The hypersurface in question is homogeneous, so in obtaining the main term of $\Sigma(X)$ one could use automorphic techniques as in the work of Duke, Rudnick and Sarnak \cite{DRS}.  However, this is of no use to us, for it would just give back the spectral formula for $\Sigma_{\mathrm{cusp}}(g_\ell,g_r)$.  

\item The only other instance that the author knows where secondary terms have been obtained via the circle method is in Vaughan and Wooley \cite{VW} and  Schindler \cite{Schindler}.
\end{remarks}

\subsection{Outline of the paper}

In \S \ref{sec-delta} we introduce our expansion of the $\delta$-symbol.  It is applied to $\Sigma(X)$ in \S \ref{sec-delta-app}.  We then apply Poisson summation in $\gamma \in \gl^{\oplus 2}_2(F)$ to the sum and then write $\Sigma(X)=\Sigma_0(X)+\Sigma^0(X)$ where $\Sigma_0(X)$ is the contribution of the $(0,0)$ term after Poisson summation and $\Sigma^0(X)$ is the contribution of the other terms (see \eqref{decomp}).  We isolate the zeroth term after Poisson summation in \S \ref{sec-vanish} and show that it is zero under assumption (ii) in the statement of Theorem \ref{main-thm}; this is the only place in the paper where this assumption is used.  

We are left with analyzing $\Sigma^0(X)$.  This requires one more application of Poisson summation (in the multiplicative sense).  The computations in the unramified case are contained in \S \ref{sec-comp}
and the
estimates required to handle the resulting sum are contained in \S \ref{sec-bounds}.  The actual application of Poisson summation in the multiplicative sense comes in \S \ref{sec-ps-in-d}, and this is where we complete the proof of Theorem \ref{main-thm}.

\subsection{Notation}
\label{ssec-notation}
Throughout this paper we use ``standard'' normalizations of Haar measures (see \cite[\S 2]{GH}).  Letting $\psi:F \backslash \A_F \to \CC^\times$ denote the ``standard'' additive character (see \cite[\S 3.1]{GH}), for $\Phi \in C_c^\infty(\gl_n(\A_F))$ we let
$$
\widehat{\Phi}(Y):=\int_{\gl_n(\A_F)}\Phi(X)\psi(\mathrm{tr}( YX)) dX
$$
denote the Fourier transform of $\Phi$.  The Poisson summation formula then takes the form
\begin{align*}
\sum_{\gamma \in \gl_n(F)}\Phi(\gamma)=\frac{1}{d_F^{n/2}}\sum_{\gamma \in \gl_n(F)}\widehat{\Phi}(\gamma)
\end{align*}
where $d_F \in \ZZ_{>0}$ is the absolute discriminant of $F$.

\section*{Acknowledgements}

The author thanks L.~Pierce, D.~Schindler and W.~Zhang for useful conversations
and H.~Hahn for her constant encouragement and help with editing.

\section{The $\delta$-symbol} \label{sec-delta}

For $m \in \OO_F^S$ let
\begin{align}
\delta^{S}(m):=\begin{cases}1 & \textrm{ if }m=0\\
0 & \textrm{ otherwise.} \end{cases}
\end{align}
Duke, Friedlander and Iwaniec \cite{DFI} introduced a very useful expression for this simple function in the case where $F=\QQ$ which has been used to great effect (see also the work of Heath-Brown \cite{HB}) and generalized to ideals of number fields in work of Browning and Vishe \cite{BV}.  We introduce a slight variant of their expression here.  It will be applied below in \S \ref{sec-delta-app}.

If $X \in \RR_{>0}$ we denote by
$$
\Delta(X) \in \A_F^{\times}
$$
the idele that is $X^{[F:\QQ]^{-1}}$ at all places $v|\infty$ and $1$ elsewhere.

In this section we prove the following proposition:
\begin{prop}  \label{prop-delta}
Let $W=\prod_{v\in S}W_v \in C_c^\infty(F_S^\times)$ be nonnegative and satisfy $\widehat{W}(0)=1$.  If $h(x,y):=W(x)-W(y/x)$ then 
 for all sufficiently large $Q \in \RR_{>0}$ one has
$$
\delta^S(m)=\frac{c_Q}{Q} \sum_{d \in \OO_F^S-0} \one_{d \widehat{\OO}_F^S}(m) h\left(\frac{d}{\Delta(Q)},\frac{m}{\Delta(Q)^{2}} \right)
$$
where for any $N>0$ one has
$$
c_Q=\sqrt{d}_F+O_N(Q^{-N}).
$$
\end{prop}

\begin{proof}

One has
\begin{align}
\sum_{\substack{d \in \OO_F^S-0\\d|m}}\left( W\left(\frac{d}{\Delta(Q)}\right)-W\left( \frac{m}{d\Delta(Q)}\right)\right)=\begin{cases} 0 &\textrm{ if }m \neq 0\\
\sum_{d \in \OO_F^S} W\left(\frac{d}{\Delta(Q)} \right) &\textrm{ if }m =0\end{cases}
\end{align}
where the first sum is over all $d \in \OO_F^S$ dividing $m$.  This is an infinite set if $F$ is not $\QQ$ or an imaginary quadratic field, but only finitely many of these $d$ yield a nonzero summand for each $m$ and $Q$.  If $Q$ is sufficiently large, then $\sum_{d \in \OO_F^S} W\left(\frac{d}{\Delta(Q)} \right) \neq 0$, and we
define
$$
c_Q:=Q \left(\sum_{d \in \OO_F^S}W\left( \frac{d}{\Delta(Q)}\right)\right)^{-1}.
$$
It is then clear that the stated identity for $\delta^S(m)$ holds.  We are left with proving the bound for $c_Q$.  By Poisson summation, one has
$$
\sum_{d \in \OO_F^S}W\left( \frac{d}{\Delta(Q)}\right)=\frac{Q}{\sqrt{d}_F}\sum_{d \in \OO_F^{S}} \widehat{W}(\Delta(Q)d).
$$
Integration by parts now yields the stated asymptotic for $c_Q$.
\end{proof}

For the remainder of the paper we make the assumption that the $W$ in the proposition satisfies
\begin{align} \label{h-assump}
W_v=\one_{\OO_{F_v}^\times}
\end{align}
for all $v \in S-\infty$.  This makes the considerations of \S \ref{sec-vanish} simpler.

\section{First manipulations with the geometric side}
\label{sec-delta-app}

We now use the notation of \S \ref{ssec-formula}.
Recall that $V_1 \in C_c^\infty((0,\infty))$. Choose $V_2 \in C_c^\infty((0,\infty))$ such that $V_2$ is identically $1$ on the support of $V_1$ and $V_3=\prod_{v \in S}V_{3v} \in C_c^\infty(F_S^\times)$ such that $V_{3v}$ is identically $1$ on a neighborhood of $1$ in $F_v^\times$ for $v|\infty$ and $V_{3v}=\one_{\OO_{F_v}^{\times}}$ for $v \in S-\infty$. 
Write
$$
V(x_1,x_2):=V_1(|x_1|_S)V_2(|x_2|_S)V_3(x_2/x_1).
$$

For $\delta^S$ as in Proposition \ref{prop-delta} one has
\begin{align*}
\Sigma(X)&=\sum_{\gamma=(\gamma_1,\gamma_2)}\sum_{b =(b_1,b_2)}\frac{V(b\det \gamma /X)}{|b_1\det \gamma_1|_SX} \delta^S(P(b,\gamma))\one_{\mathcal{F}}( b \det g_\ell g_r^{-1}) f\one_{\gl_2(\widehat{\OO}_F^S)^{\oplus 2}}(g^{-1}_\ell \gamma g_r)
\end{align*}
where the sums on $\gamma$ and $b$ are over $\GL_2(F)^{\times 2}$ and $(F^\times)^{\oplus 2}$, respectively.
Here $\one_{\mathcal{F}}$ is defined as in assumption (vi).
We observe that the presence of the $V_2$ and $V_3$ in the definition of $V$ is redundant, but it simplifies 
matters when we later apply Poisson summation in \S \ref{sec-ps-in-d} (compare Proposition \ref{prop-arch} and Corollary \ref{cor-h}).  We also note that by definition of $V$ the sum on $b$ is finite in a sense depending only on $g_\ell, g_r$, $V$, $\mathcal{F}$ and $f$.

Applying Proposition \ref{prop-delta} with $Q=\sqrt{X}$ we obtain
\begin{align*}
\Sigma(X)= \sum_{\gamma,b}\frac{V(b\det \gamma/X)}{|b_1\det \gamma_1|_SX} \frac{c_{\sqrt{X}}}{\sqrt{X}}&\sum_{d \in \OO_F^S-0}\one_{d \widehat{\OO}_F^S}(P(b,\gamma))h\left(\frac{d}{\Delta(\sqrt{X})},\frac{P(b,\gamma)}{\Delta(X)} \right)\\&\times\one_{\mathcal{F}}(b\det g_\ell g_r^{-1}) f\one_{\gl_2(\widehat{\OO}_F^S)^{\oplus 2}}(g^{-1}_\ell \gamma g_r).
\end{align*}
\begin{rem} Notice that in the sum above the moduli $d$ satisfy $|d|_v \ll \sqrt{X}^{[F:\QQ]^{-1}}$ for all $v|\infty$ and the $v$-norm of the entries of $\gamma_1,\gamma_2$ is also bounded by $O(\sqrt{X}^{[F:\QQ]^{-1}})$ for $v|\infty$.  Thus it is reasonable to expect that one can shorten the length of the sum by applying Poisson summation in $\gl_2(F) \times \gl_2(F)$.  This is indeed the case.
\end{rem}

 We apply Poisson summation  in  $\gamma=(\gamma_1,\gamma_2) \in \gl_2(F)^{\oplus 2}$ (see \S \ref{ssec-notation}) to arrive at 
\begin{align} \label{after-first-ps}
\Sigma(X)&=d_F^4\sum_{\gamma \in \gl_2(F)^{\oplus 2}}\sum_{b}\int_{\gl_2(\A_F)^{\oplus 2}}\frac{V(b\det T/X)}{|b_1\det T_1|_SX} \frac{c_{\sqrt{X}}}{\sqrt{X}}\sum_{d \in \OO_F^S-0}\one_{d \widehat{\OO}_F^S}(P(b,T))\\& \times h\left(\frac{d}{\Delta(\sqrt{X})},\frac{P(b,T)}{\Delta(X)} \right)\one_{\mathcal{F}}(b\det g_\ell g_r^{-1}) \nonumber  f\one_{\gl_2(\widehat{\OO}_F^S)^{\oplus 2}}(g_\ell^{-1} T g_r) \psi\left(\frac{\mathrm{tr}\, \gamma T}{d}\right)dT. \nonumber
\end{align}
Here $T=(T_1,T_2)$, $dT=dT_1dT_2$ is the Haar measure on $\gl_2(\A_F)^{\oplus 2}$. 
It is convenient to write
\begin{align} \label{decomp}
\Sigma(X)=\Sigma_0(X)+\Sigma^0(X),
\end{align}
where $\Sigma_0(X)$ is the contribution of the $\gamma=(0,0)$ term and $\Sigma^0(X)$ is the contribution of the terms with $\gamma \neq (0,0)$.  We will show in \S \ref{sec-vanish} below that $\Sigma_0(X)$ vanishes under a
mild assumption.

To complete our analysis of $\Sigma^0(X)$ we will apply Poisson summation in $d \in F^\times$ in \S \ref{sec-ps-in-d}.  Before doing this we collect the necessary local computations and bounds in \S \ref{sec-comp} and \S \ref{sec-bounds}, respectively.

\section{Vanishing of $\Sigma_0(X)$} \label{sec-vanish}

Let $v \in S-\infty$ and let $k \in \ZZ_{ > 0}$.  For $\Phi_v \in C_c^\infty(\GL_2(F_v))$ recall assumption (A) of \S \ref{ssec-formula}:
\begin{itemize}
\item[(A)] The function $\Phi_v$ is supported in the set of $g$ with $v(\det g)=k$,\\ $\int_{\GL_2(F_v)}\Phi_v(g)dg=0$, and $\Phi_v \in C_c^\infty(\GL_2(F_v)//\GL_2(\OO_{F_v}))$.
\end{itemize}

It is clear that if $\pi_v$ is an abelian twist of the trivial representation of $\GL_2(F_v)$ and $\Phi_v$ satisfies assumption (A) then $\pi_v(\Phi_v)=0$.  On the other hand, we have the following lemma:

\begin{lem} \label{lem-no-loss}
If $\pi_v$ is the local factor of a generic unitary automorphic representation of $ \GL_2(\A_F)$ unramified at $v$ then there exists a $\Phi_v$ satisfying assumption (A) such that $\pi_v(\Phi_v) \neq 0$.
\end{lem}

\begin{proof} Let $\varpi_v$ be a uniformizer for $F_v$ and let $q_v:=|\varpi_v|^{-1}$. Consider  
$$
\Phi_v:=\one_{\varpi_v^2}-(q^2_v+q_v+1)\one_{\varpi_v\GL_2(\OO_{F_v})}.
$$
It is a standard result that
$$
\bigcup_{g \in \gl_2(\OO_{F_v}):\,\det g\OO_{F_v}^\times=\varpi^2_v\OO_{F_v}^\times}
\GL_2(\OO_{F_v})g\GL_2(\OO_{F_v})
$$
can be written as a disjoint sum of $q^2_v+q_v+1$ elements of $\GL_2(F_v)/\GL_2(\OO_{F_v})$ \cite[Proof of Proposition 4.4]{KL}.  Therefore $\Phi_v$ satisfies assumption (A) with $k=2$.  

Let $\alpha,\beta \in \CC^\times$ be the Satake parameters of $\pi_v$.  Then $\pi_v(\Phi_v)$ projects the space of $\pi_v$ to the spherical vector and acts via the scalar
\begin{align}
q_v(\alpha^2+\alpha \beta +\beta^2) -(q^2_v+q_v+1)\alpha\beta
\end{align}
on this vector  \cite[Proposition 4.4]{KL}.  If this quantity is zero, then $\alpha/\beta$ is a root of the polynomial $q_vx^2-(q_v^2+1)x+q_v$.  But this is impossible, 
for $|\alpha \beta|=1$ and $|\alpha|,|\beta|<q^{1/2}_v$ since $\pi_v$ is unitary and generic \cite[(2.5)]{JS}.

\end{proof}

Now assume that our test function $f \in C_c^\infty((A \backslash \GL_2(F_S))^{\times 2})$ satisfies
$f=f^v \otimes f_{1v}f_{2v}$ with $f^v \in C_c^\infty((A \backslash \GL_2(F_{S-v}))^{\times 2})$ and $f_{1v},f_{2v} \in C_c^\infty(\GL_2(F_v))$.  Assume moreover that $f_{1v}$ satisfies assumption (A) for a given $k>0$, and that $f_{2v}=\one_{\GL_2(\OO_{F_v})}$.

\begin{rem} By Lemma \ref{lem-no-loss} this assumption is essentially no loss of generality for the purpose of applying our main result, Theorem \ref{main-thm}.
\end{rem}

\begin{lem} \label{lem-vanish} Under the above assumptions $\Sigma_0(X)=0$.
\end{lem}

\begin{proof}
Since $h(x,y)=W(x)-W(y/x)$ it suffices to check that 
\begin{align} \label{first}
\int_{\gl_2(F_v)^{\oplus 2}}  f_{1v}f_{2v}(g_\ell^{-1} T g_r)dT=\int_{\gl_2(F_v)^{\oplus 2}}  f_{1v}\one_{\GL_2(\OO_{F_v})}(g_\ell^{-1} T g_r)dT
\end{align}
and
\begin{align} \label{second}
&\int_{\gl_2(F_v)^{\oplus 2}} W_v\left(\frac{P(b,T)}{t}\right)\one_{\mathcal{F}}(b\det g_\ell g_r^{-1})  f_{1v}f_{2v}(g_\ell^{-1} T g_r)dT\\
&=|\det g_\ell g_r^{-1}|^2_v\int_{\gl_2(F_v)^{\oplus 2}} \one_{\OO_{F_v}^{\times}}\left(\frac{P(b \det g_\ell g_r^{-1},T)}{t}\right)\one_{\mathcal{F}}(b\det g_\ell g_r^{-1})  f_{1v}\one_{\GL_2(\OO_{F_v})}(T )dT \nonumber
\end{align}
are both zero for any $t \in F_v^\times$.  We note that on the support of $f_{1v}\one_{\GL_2(\OO_{F_v})}$ the measure $dT$ is a scalar multiple of the multiplicative Haar measure $dg$.  Thus 
by assumption (A) it is clear that \eqref{first} vanishes.  

On the other hand it not hard to see that the integrand in \eqref{second} is nonzero only if $t \in \OO_{F_v}^{\times}$, and in this case $\one_{\OO_{F_v}^{\times}}\left(\frac{P(b \det g_\ell g_r^{-1},T)}{t}\right)\one_{\mathcal{F}}(b\det g_\ell g_r^{-1})$ is identically $1$ on the support of $f_{1v}\one_{\GL_2(\OO_{F_v})}$.  It follows that \eqref{second} is zero.

\end{proof}

\section{Nonarchimedian computations} \label{sec-comp}

Let $v\not \in S$ be a nonarchimedian place of $F$.  We work locally in this section and drop the subscript $v$, writing $F:=F_v$.  We let $\varpi$ be a uniformizer for $F$ and let $q=|\varpi|^{-1}$.  We assume $b_1,b_2 \in \OO_F^\times$.  
In this section we compute
\begin{align} \label{na-int}
\int_{\OO_F}\int_{\gl_2^{\oplus 2}(\OO_F)}\one_{t \OO_{F}}(P(b,T))\psi\left(\frac{\mathrm{tr}\,\gamma T}{t}\right)dT
\chi(t)|t|^sdt^\times.
\end{align}
The main result is Proposition \ref{prop-na-comp} below.

\begin{lem} \label{lem-DF} Assume that $t \in \OO_F$ with $v(t)>1$.  If $\gamma_0 \in \gl_2(\OO_F)$ and $x \in \OO_F^\times$ then
\begin{align*}
\int_{\gl_2(\OO_F)} \psi\left(\frac{x \det T+\mathrm{tr}\,\gamma_0 T}{t}\right)dT=|t|^2\psi\left(\frac{- \det \gamma_0}{xt} \right).
\end{align*}

\end{lem}
\begin{proof} Let
\begin{align*}
\mathcal{F}_{x}(T):=&x\det T+\mathrm{tr}\, \gamma_0 T.
\end{align*}
Let $p$ be the rational prime below $v$ and let 
$$
D_x:=\mathrm{Res}_{\OO_F/\ZZ_p}\mathrm{Spec}(\OO_F[T]/(\nabla \mathcal{F}_x)) \subset
\A_{\ZZ_p}^{4[F\,:\,\QQ_p]} \quad \textrm{ (affine }4[F:\QQ_p]\textrm{ space).}
$$
Here $(\nabla \mathcal{F}_x) \leq \OO_F[T]$ is the ideal generated by the entries of 
$\nabla \mathcal{F}_x$.  
Writing $T=(t_{ij})$ and $\gamma_0=(\gamma_{ij})$ one has
\begin{align*}
\nabla \mathcal{F}_{x}(T)=\begin{pmatrix} \gamma_{11}+x t_{22} \\
\gamma_{12}-x t_{12}\\
\gamma_{21}-x t_{21} \\
\gamma_{22}+x  t_{11}\end{pmatrix}.
\end{align*}
The Hessian is
\begin{align*}
H_{T}(x)&=x \begin{pmatrix} & & & 1 \\ & & -1 &\\ & -1 & &\\ 1 & & &  \end{pmatrix}.
\end{align*}
We note that since $\varpi \nmid x$, the Hessian $H_T(x)$ is nonvanishing on $D_x(\ZZ_p/p)$, and we conclude that $D_x$ is \'etale over $\ZZ_p$.  Applying \cite[Theorem 1.4]{DF} (a result which the authors attribute to Katz) we have that 
\begin{align} \label{fixed-x22}
\int_{\gl_2(\OO_F)} \psi\left(\frac{x \det T+\mathrm{tr}\,\gamma_0 T}{t}\right)dT=|t|^2\sum_{T \in D_x(\ZZ_p)}\psi\left(\frac{\mathcal{F}_x(T)}{t} \right)G_t(H_T(x))
\end{align}
where 
\begin{align*}
G_t(H_T(x))=\begin{cases}1 & \textrm{ if } 2|v(t)\\
q^{-2}\sum_{X \in (\OO_F/\varpi)^{4}} \psi\left( \frac{X^t\tfrac{1}{2}H_T(x)X}{\varpi}\right)& \textrm{ if } 2 \nmid v(t).\end{cases}
\end{align*}
Here some care is required in interpreting the transpose; we must view $X$ as a vector in $(\OO_F/\varpi)^{4}$ with respect to the basis we used in writing down the Hessian. It is not hard to directly compute that $G_t(H_T(x))=1$.  Therefore \eqref{fixed-x22} becomes
\begin{align} \label{fixed-x32}
|t|^2\sum_{T \in D_x(\ZZ_p)}\psi\left(\frac{\mathcal{F}_x(T)}{t} \right).
\end{align}
The set $D_x(\ZZ_p)$ contains only one element, namely
$$
D_x(\ZZ_p)=
\left\{\frac{1}{x}\begin{pmatrix} -\gamma_{22} & \gamma_{12} \\ \gamma_{21}& -\gamma_{11}\end{pmatrix}\right\}
$$
and thus
\begin{align*}
\int_{\gl_2(\OO_F)}\psi\left(\frac{x \det T+\mathrm{tr}\,\gamma_0 T}{t}\right)dT=|t|^2\psi\left( \frac{-\det \gamma_0}{xt}\right).
\end{align*}
\end{proof}

The following is the analogue for $v(t)=1$:

\begin{lem} \label{lem-Gaussian}  If $\gamma_0 \in \gl_2(\OO_F)$ and $x \in \OO_F^\times$ then
\begin{align*}
\int_{\gl_2(\OO_F)} \psi\left(\frac{x \det T+\mathrm{tr}\,\gamma_0 T}{\varpi}\right)dT=q^{-2}\psi\left(\frac{- \det \gamma_0}{x\varpi} \right).
\end{align*}

\end{lem}

\begin{proof}

We begin with the case $\gamma_0=0$.  In this case we compute
\begin{align*}
&\int_{\gl_2(\OO_F)} \psi\left( \frac{x \det T}{\varpi}\right)dT
\\&=q^{-4}\sum_{g \in \GL_2(\OO_F/\varpi)} \psi\left( \frac{x \det g}{\varpi} \right)+q^{-4} \sum_{\substack{T \in \gl_2(\OO_F/\varpi)\\ \det T=0}}1\\
&=q^{-4}\left(\sum_{g \in \SL_2(\OO_F/\varpi)} \sum_{\alpha \in (\OO_F/\varpi)^\times}\psi\left(\frac{x\alpha}{\varpi} \right) +q^4-(q^2-1)(q^2-q) \right)\\
&=q^{-4}\left( -\frac{(q^2-1)(q^2-q)}{q-1}+q^3+q^2-q\right)=q^{-2}.
\end{align*}

For the general case, let $X=\left( \begin{smallmatrix} x_1 & x_2 \\ x_3 & x_4 \end{smallmatrix}\right)$, $T=\left( \begin{smallmatrix} t_1 & t_2 \\ t_3 & t_4 \end{smallmatrix} \right)$.  Then we start by observing that 
$$
\det (T+X)=\det T+\det X+\mathrm{tr}(f(X)T)
$$
where $f$ is the $\OO_F$-linear isomorphism
\begin{align*}
f:\gl_2(\OO_F) &\lto \gl_2(\OO_F)\\
X &\longmapsto \left( \begin{smallmatrix}  x_4&-x_2\\-x_3 &x_1\end{smallmatrix}\right).
\end{align*}
Taking $x=1$ and $X=f^{-1}(\gamma_0)$
we therefore have
\begin{align*}
&\int_{\gl_2(\OO_F)} \psi\left(\frac{ \det T+\mathrm{tr}\,\gamma_0 T}{\varpi}\right)dT\\&=\int_{\gl_2(\OO_F)} \psi\left(\frac{ \det (T-f^{-1}(\gamma_0))+\mathrm{tr}\,\gamma_0 T-\mathrm{tr}\, \gamma_0 f^{-1}(\gamma_0)}{\varpi}\right)dT
\nonumber \\
&=\int_{\gl_2(\OO_F)} \psi\left(\frac{ \det T+ \det f^{-1}(\gamma_0)-\mathrm{tr}\, \gamma_0 f^{-1}(\gamma_0)}{\varpi}\right)dT \nonumber\\
&=q^{-2}\psi\left(\frac{-\det \gamma_0}{\varpi} \right).
\end{align*}
The general case (i.e. when $x \neq 1$) follows from a change of variables $\gamma_0 \mapsto x^{-1}\gamma_0$.
\end{proof}

\begin{prop} \label{prop-na-comp} The integral
\begin{align*}
\int_{\OO_F}\int_{\gl_2^{\oplus 2}(\OO_F)}\one_{t \OO_{F}}(P(b,T))\psi\left(\frac{\mathrm{tr}\,\gamma T}{t}\right)dT
\chi(t)|t|^sdt^\times
\end{align*}
is equal to 
\begin{align*}
\int_{\OO_F} \chi(c)|c|^s\one_{gl_2^{\oplus 2}(\OO_F)}(c^{-1}\gamma)(1-\chi(\varpi)q^{-s-5})\int_{\OO_F}
\one_{t\OO_F}(P(b^{-1},c^{-1}\gamma))\chi(t)|t|^{s+4}dt^\times dc^\times.
\end{align*}

\end{prop}

\begin{proof}
We start by writing
\begin{align*}
&\int_{\OO_F}\int_{\gl_2^{\oplus 2}(\OO_F)}\one_{t \OO_{F}}(P(b,T))\psi\left(\frac{\mathrm{tr}\,\gamma T}{t}\right)dT
\chi(t)|t|^sdt^\times\\
&=\int_{\OO_F}\int_{\gl_2^{\oplus 2}(\OO_F)}
\int_{\OO_F}\psi\left(
\frac{xP(b,T)+\mathrm{tr}\,\gamma T}{t}\right)dx dT
\chi(t)|t|^sdt^\times\\
&=\int_{\OO_F} \sum_{c|t}|c|\int_{\gl_2^{\oplus 2}(\OO_F)} \int_{\OO_F^\times} \psi\left(\frac{xcP(b,T)+\mathrm{tr}\,\gamma T}{t} \right)dx dT
\chi(t)|t|^sdt^\times
\end{align*}
where the sum on $c$ is over divisors of $t$ in $\OO_F$.
We take a change of variables $t \mapsto ct$ to arrive at 
\begin{align*}
\sum_{c} \chi(c)|c|^{s+1}\int_{\OO_F}\int_{\OO_F^\times}\int_{\gl_2^{\oplus 2}(\OO_F)}  \psi\left(\frac{xP(b,T)+\mathrm{tr}\,c^{-1}\gamma T}{t} \right)dT dx \chi(t)|t|^sdt^\times
\end{align*}
where the sum on $c$ is over representatives for the ideals of $\OO_F$.
We now invoke lemmas \ref{lem-DF} and \ref{lem-Gaussian} to write the above as 
\begin{align*}
&\sum_{c} \chi(c)|c|^{s+1}\one_{\gl_2^{\oplus 2}(\OO_F)}(c^{-1}\gamma)\int_{\OO_F}\int_{\OO_F^\times}
  \psi\left(\frac{x(-b_1^{-1}\det c^{-1}\gamma_1+b_2^{-1}\det c^{-1}\gamma_2)}{t} \right)dT dx \chi(t)|t|^{s+4}dt^\times \\
  &=\sum_{c} \chi(c)|c|^s\one_{\gl_2^{\oplus 2}(\OO_F)}(c^{-1}\gamma)\\& \times \left(\int_{\OO_F}
\one_{t\OO_F}(P(b^{-1},c^{-1}\gamma))\chi(t)|t|^{s+4}dt^\times-
q^{-1}\int_{\varpi\OO_F}\one_{t\varpi^{-1}}(P(b^{-1},c^{-1}\gamma))\chi(t)|t|^{s+4}dt^\times\right)\\
&=\int_{\OO_F} \chi(c)|c|^s\one_{\gl_2^{\oplus 2}(\OO_F)}(c^{-1}\gamma)(1-\chi(\varpi)q^{-s-5})\int_{\OO_F}
\one_{t\OO_F}(P(b^{-1},c^{-1}\gamma))\chi(t)|t|^{s+4}dt^\times dc^\times.
\end{align*}

\end{proof}

\section{Bounds for local integrals} \label{sec-bounds}

In this section we collect the rough bounds on local integrals we require to analyze $\Sigma(X)$ in \S \ref{sec-ps-in-d} below.
For the remainder of this paper, if $v$ is a place of $F$  we set
$$
|\left(\begin{smallmatrix}\gamma_{11} & \gamma_{12} \\ \gamma_{21} & \gamma_{22} \end{smallmatrix}\right)|_v:=\mathrm{max}(|\gamma_{ij}|_v).
$$
Moreover, if $\gamma =(\gamma_1,\gamma_2)\in \gl_2^{\oplus 2}(F_v)$ then we set
$|\gamma|_v:=\max(|\gamma_1|_v,|\gamma_2|_v)$.

\subsection{Archimedian integrals}

For this subsection we work locally at a fixed place $v|\infty$ and omit it from notation, writing $F:=F_v$, etc.
Let $W \in C_c^\infty(F^\times)$, $f \in C_c^\infty(\gl_2^{\oplus 2}(F))$, $b=(b_1,b_2) \in (F^\times)^{\oplus 2}$ and assume
$$
|b_1| \asymp |b_2| \asymp 1
$$
for all $v|\infty$.  All implied constants in this section are allowed to depend on $W$, $f$, and the bounds on $b$.
Let
$$
\chi:F^\times \lto \CC^\times
$$
be a (unitary) character, and let $s \in \CC$.  
For $t \in \RR$ we let $C(\chi,t)$ be the analytic conductor of $\chi$ normalized as in \cite[\S 1]{Brumley}.

\begin{prop} \label{prop-arch}
Let $h_0(x,y)$ be either $W(x)$ or $W(y/x)$.
Consider the integral
\begin{align*}
\int_{F^\times}\left(\int_{\gl_2(F)^{\oplus 2}}
h_0(t,P(b,T))
f(T) \psi\left(\frac{\mathrm{tr}\,\gamma T}{t}\right)dT\right)\chi(t)|t|^sdt^\times.
\end{align*} 
Let
$N \in \ZZ_{ \geq 0}$,  $1>\varepsilon>0$, $\lambda>\varepsilon$.

If $\gamma =(0,0)$ then for any  and $s \in \CC$ with $\lambda >\mathrm{Re}(s)>\varepsilon-1$ the integral is $O_{N,\varepsilon,\lambda}(C(\chi,\mathrm{Im}(s))^{-N})$.

If $\gamma \neq (0,0)$ and $s \in \CC$ with $\lambda >\mathrm{Re}(s)>\varepsilon-4$ then the integral 
is
$$
O_{N,\lambda,\varepsilon}\left(\mathrm{max}(
|\gamma|,1)^{-N} C(\chi,\mathrm{Im}(s))^{-N}|\gamma|^{ 12-64/\varepsilon }\right).
$$

\end{prop}

The crucial step in the proof is the following lemma which we isolate for future use:
\begin{lem} \label{lem-arch-step} Let $N \in \ZZ_{\geq 0}, 1 >\varepsilon>0$, $\lambda>\varepsilon$.  If $\gamma \neq (0,0)$ then the integral
\begin{align*}
\int_{\gl_2(F)^{\oplus 2}}
W\left(\frac{P(b,T)}{t}\right)
f(T) \psi\left(\frac{\mathrm{tr}\,\gamma T}{t}\right)dT
\end{align*}
is bounded by a constant depending on $\varepsilon,N$ times
$$
\mathrm{max}( |\gamma|,1)^{-N}\mathrm{min}(|t|^{4-\varepsilon},|t|^{-N})|\gamma|^{6-32/\varepsilon }.
$$
\end{lem}

\begin{proof}

Assume first that $|t| > \max(|\gamma|,1)^{1/4}$.
In this case the integrand in the lemma is supported in the set of $T$ such that $|T|^2 \gg |t| >\max(|\gamma|,1)^{1/4}$.  Since $f$ is Schwartz, this implies the desired bound.  

We henceforth assume that $|t| \leq \max(|\gamma|,1)^{1/4}$.
By Fourier inversion we obtain
\begin{align}
&\int_{\gl_2(F)^{\oplus 2}}
W\left(\frac{P(b,T)}{t}\right)
f(T) \psi\left(\frac{\mathrm{tr}\,\gamma T}{t}\right)dT \nonumber \\
&=\int_{F \times \gl_2(F)^{\oplus 2}}
\widehat{W}\left(x\right)
f(T) \psi\left(\frac{\mathrm{tr}\,\gamma T-xP(b,T)}{t}\right)dxdT.\label{to-bound}
\end{align}
Write $T_1:=(t_{1,ij})$ and $T_2:=(t_{2,ij})$ and similarly for $\gamma=(\gamma_1,\gamma_2)$.  Assume that $|\gamma_{1,11}|=\mathrm{max}( |\gamma|,1)$.  We apply integration by parts in $t_{1,11}, t_{1,12},t_{2,11},t_{2,12}$ to see that for any $\varepsilon>0$ and $N' \in \ZZ_{>0}$ the integral \eqref{to-bound} is equal to $O_{N',\varepsilon}(|t|^{N'\varepsilon/8}|\gamma_{1,11}|^{-N'/2})$ plus 
\begin{align} \label{to-bound2}
O\left(\int |\widehat{W}(x)f(T)| dxdT \right)
\end{align}
where the integral is over the set of $x,T$ such that
\begin{align} \label{ineq}
|\gamma_{1,11}-xb_1t_{1,22}|,|\gamma_{1,21}+xb_1t_{1,21}|,|\gamma_{2,11}+xb_2t_{2,22}|,|\gamma_{2,22}-xb_2t_{1,21}| \leq |t|^{1-\varepsilon/8}|\gamma_{1,11}|^{1/2}.
\end{align}
Since $f(T)$ is Schwartz, we can further truncate the integral to the subset satisfying 
\begin{align} \label{ineq-2}
|t_{1,22}| \ll |\gamma_{1,11}|^{1/2}|t|^{-\varepsilon/8}
\end{align}
at the expense of introducing an error of $O_{N}(|t|^{N'\varepsilon/8}|\gamma_{1,11}|^{-N'/2})$.

Notice that $|t|^{1-\varepsilon/8}|\gamma_{1,11}|^{1/2} \ll |\gamma_{1,11}|^{3/4-\varepsilon/32}$.
Thus if $|\gamma_{1,11}|$ is sufficiently large, then the first inequality of \eqref{ineq} implies that 
 $|\gamma_{1,11}| \ll |xt_{1,22}|$ which, by \eqref{ineq-2} implies that 
 $$
 |\gamma_{1,11}|^{1/2}|t|^{\varepsilon/8} \ll |x|.
 $$
If, on the other hand, $\gamma_{1,11}$ lies in a compact subset of $F^\times$
 and $|t|$ is sufficiently small then the first inequality implies $1 \ll |xt_{1,22}|$ where the implied constant depends on the compact set.  This implies in view of \eqref{ineq-2} that $|t|^{\varepsilon/8} \ll |x|$. 
 Thus bounding the integrals in \eqref{to-bound2} trivially and using the fact that $\widehat{W}$ and $f$ are rapidly decreasing  we see that for any $N' \in \ZZ_{ \geq 0}$ the integral \eqref{to-bound} is bounded by a constant times
$$
|t|^{4-\varepsilon}|\gamma_{1,11}|^{-N'}.
$$
An analogous argument handles the cases where $\max(|\gamma|,1)$ is the norm of another matrix entry of $\gamma_1$ or $\gamma_2$.  

Assume now that $\mathrm{max}( |\gamma|,1)=1$.  
Assume moveover that $|\gamma_{1,11}|=\mathrm{max}(|\gamma_1|,|\gamma_2|)$; this is bigger than zero since $\gamma \neq (0,0)$.

Applying integration by parts as above we see that for any $\varepsilon>0$ and $N' \in \ZZ_{>0}$ the integral \eqref{to-bound} is equal to $O_{N',\varepsilon}(|t|^{N'\varepsilon/8}|\gamma_{1,11}|^{-2N'})$
plus the integral \eqref{to-bound2} taken over the set of $x,T$ such that 
\begin{align} \label{domain-1}
|\gamma_{1,11}-xb_1t_{1,22}|,|\gamma_{1,21}-xb_1t_{1,21}|,|\gamma_{2,11}+xb_2t_{2,22}|,|\gamma_{2,21}-xb_2t_{1,21}| \leq |t|^{1-\varepsilon/8}|\gamma_{1,11}|^{2}.
\end{align}

Since $f(T)$ is Schwartz, we can further truncate the integral to the subset satisfying 
\begin{align} \label{ineq-4}
|t_{1,22}| \ll |\gamma_{1,11}|^{2}|t|^{-\varepsilon/8}
\end{align}
at the expense of introducing an error of $O_{N}(|t|^{N'\varepsilon/8}|\gamma_{1,11}|^{-2N'})$

The first inequality in \eqref{domain-1} implies that for $|\gamma_{1,11}|$ sufficiently small we must have 
$|xt_{1,22}| \gg |\gamma_{1,11}|$, which implies in view of \eqref{ineq-4} that 
$$
|x| \gg |t|^{\varepsilon/8}|\gamma_{1,11}|^{-1}.
$$
If $|\gamma_{1,11}|$ is bounded away from zero and $|t|$ is sufficiently small we have $|xt_{1,22}| \gg 1$, where the implied constant depends on the bound on $|\gamma_{1,11}|$.  In this case we conclude via \eqref{ineq-4} that
$$
|x| \gg |t|^{\varepsilon/8}|\gamma_{1,11}|^{-2}.
$$
In either case we conclude that the integral in \eqref{to-bound2} taken over the domain defined by \eqref{domain-1} and \eqref{ineq-4} is bounded by a constant times
$$
|t|^{4-\varepsilon}|\gamma_{1,11}|^{12}.
$$
Taking $N'=\lceil \frac{32}{\varepsilon}-8\rceil$ and noting that $0 >-2\lceil \frac{32}{\varepsilon}-8 \rceil \geq 12-\frac{64}{\varepsilon}$ we arrive at a bound of $O(|t|^{4-\varepsilon}|\gamma_{1,11}|^{12-64/\varepsilon})$ for \eqref{to-bound} in this case.  An analogous argument establishes the same bound if $|\gamma|$ is less than $1$ and equal to the norm of some other matrix entry of $\gamma_1$ or $\gamma_2$.

\end{proof}

We now prove Proposition \ref{prop-arch}:

\begin{proof}[Proof of Proposition \ref{prop-arch}]
  The assertions when $h_0(x,y)$ is $W(x)$ are clear, for $W(t)f(tT)$ is Schwartz and compactly supported as a function of $(t,T) \in F^\times \times \gl_2(F)^{\oplus 2}$.  Assume that $h_0(x,y)=W(y/x)$.

Assume first that $\gamma=(0,0)$.  We then take a change of variables $t \mapsto P(b,T)t$ to arrive at 
\begin{align*}
&\int_{F^\times}\left(\int_{\gl_2(F)^{\oplus 2}}
W\left(\frac{P(b,T)}{t}\right)
f(T) dT\right)\chi(t)|t|^sdt^\times\\
&=\int_{F^\times}
W\left(\frac{1}{t}\right)\chi(t)|t|^sdt^\times
\int_{\gl_2^{\oplus 2}(F)}f(T) |P(b,T)|^s\chi(P(b,T))dT.
\end{align*}
The integral $\int_{F^\times}
W\left(\frac{1}{t}\right)\chi(t)|t|^sdt^\times$ is clearly $O_{N,\varepsilon,\lambda}(C(\chi,\mathrm{Im}(s))^{-N})$.  As for the latter factor one has
\begin{align} \label{latter}
\left|\int_{\gl_2^{\oplus 2}(F)}f(T)|P(b,T)|^s\chi(P(b,T))dT \right|\leq \int_{\gl_2^{\oplus 2}(F)}|f(T)||P(b,T)|^{\mathrm{Re}(s)}dT.
\end{align}
We need only show this converges for $\mathrm{Re}(s)>-1$.  
Write $P(b,\cdot)_*f:F \to \CC$ for the push-forward of $f(T)$ along the smooth surjection 
\begin{align*} 
P(b,\cdot):\gl_2^{\oplus 2}(F) \lto F\\
T \mapsto P(b,T).
\end{align*}
This map is in fact submersive away from $(0,0)$, so $P(b,\cdot)_*f$ is again a Schwartz function on $F$.  
Thus the integral \eqref{latter} is
\begin{align}
\int_{F} P(b,\cdot)_*f(t)|t|^{\mathrm{Re}(s)}dt
\end{align}
which converges absolutely for $\mathrm{Re}(s)>-1$, as desired.

Let $D=t\frac{\partial}{\partial t}$ (and if $v$ is complex, $\bar{D}=\bar{t} \frac{\partial}{\partial \bar{t}}$), viewed as differential operators on $F^\times$.  We claim that for any $i \geq 0$ (and $j \geq 0$ if $v$ is complex), $\varepsilon>0$ and $N \geq 0$ one has
\begin{align} \label{basic-bound}
& D^i\bar{D}^j\left(\int_{\gl_2(F)^{\oplus 2}}
W\left(\frac{P(b,T)}{t}\right)
f(T) \psi\left(\frac{\mathrm{tr}\,\gamma T}{t}\right)dT\right)\\
& \ll_{i,j,\varepsilon,N} \mathrm{max}( |\gamma|,1)^{-N}\mathrm{min}(|t|^{4-\varepsilon},|t|^{-N})|\gamma|^{12-64/\varepsilon }. \nonumber
\end{align}
Here and in the remainder of the proof all implied constants are allowed to depend on $f,W$ and the bounds on $|b_1|,|b_2|$.
Assuming the claim, repeated application of integration by parts in $t$ (and $\bar{t}$ if $v$ is complex) implies the proposition.  On the other hand, since $W$ and $f$ are arbitrary it is not hard to see that the claim follows from the special case where $i=j=0$.  In other words, we are to show that
\begin{align} \label{int}
\int_{\gl_2(F)^{\oplus 2}}
W\left(\frac{P(b,T)}{t}\right)
f(T) \psi\left(\frac{\mathrm{tr}\,\gamma T}{t}\right)dT
\end{align}
is bounded by a constant depending on $\varepsilon,N$ times
$$
\mathrm{max}( |\gamma|,1)^{-N}\mathrm{min}(|t|^{4-\varepsilon},|t|^{-N})|\gamma|^{6-32/\varepsilon }.
$$
This is the content of Lemma \ref{lem-arch-step}.

\end{proof}

\subsection{Nonarchimedian integrals}

In this section we assume that $v$ is a nonarchimedian place of $F$ and omit it from notation, writing $F:=F_v$, etc.  We denote by $\varpi$ a uniformizer for $F$, let $q=|\varpi|^{-1}$, and let $\mathcal{D}_F$ be the absolute different of $F$.

We prove the following proposition:
\begin{prop} \label{prop-na-bound}
Let $b_1,b_2 \in F^\times$, $f \in C_c^\infty(\gl_2(F)^{\oplus 2})$.
The integral 
$$
\int_{F^\times} \one_{\OO_F}(t) \int_{\gl_2(F)^{\oplus 2}}
\one_{t\OO_F}(P(b,T))
f(T) \psi\left(\frac{\mathrm{tr}\,\gamma T}{t}\right)dT\chi(t)|t|^sdt^\times
$$
converges absolutely for $\mathrm{Re}(s)>0$.  It vanishes if $|\gamma|$, or the absolute norm of the conductor of $\chi$ is sufficiently large in a sense depending on $f$ and $|b_1|,|b_2|$.

For $t \in \OO_F-0$ consider the integral
\begin{align} \label{int-cons}
\int_{\gl_2(F)^{\oplus 2}}
\one_{t\OO_F}(P(b,T))
f(T) \psi\left(\frac{\mathrm{tr}\,\gamma T}{t}\right)dT.
\end{align}
If $(\gamma_1,\gamma_2) \neq (0,0)$ and $v(t)>1$ then \eqref{int-cons}
 is bounded in absolute value by a constant depending on $f$, $|b_1|,|b_2|$ times 
\begin{align*}
q^{4\min(v(\gamma_{1,ij}),v(\gamma_{2,i'j'}))}|t|^4.
\end{align*}

\end{prop}

\begin{proof} 
We use a Fourier transform to rewrite the integral as 
$$
\int_{F^\times} \one_{\OO_F}(t)\left( \int_{\mathcal{D}_F^{-1} \times 
\gl_2(F)^{\oplus 2}}
f(T) \psi\left(\frac{\mathrm{tr}\,\gamma T-xP(b,T)}{t}\right)dT dx\right)\chi(t)|t|^sdt^\times.
$$
We assume without loss of generality that 
$$
f=\one_{\beta\varpi^{-m}+\varpi^{k}\gl_2(\OO_F)^{\oplus 2}}
$$
for some $\beta=(\beta_1,\beta_2) \in \gl_2(\OO_F)^{\oplus 2}$
and $m,k \geq 0$; thus the above becomes 
\begin{align}  \label{before-q}
&\int_{F^\times}\one_{\OO_F}(t)\Bigg( \int_{\mathcal{D}_F^{-1} \times \gl_2(F)^{\oplus 2}}
\one_{\varpi^k\gl_2(\OO_F)^{\oplus 2}}(T-\beta \varpi^{-m}) \psi\left(\frac{\mathrm{tr}\,\gamma T-xP(b,T)}{t}\Bigg)dTdx\right)\chi(t)|t|^sdt^\times. 
\end{align}
It is clear that the multiple integral over $F^\times \times \mathcal{D}_F^{-1} \times \gl_2(F)^{\oplus 2}$ converges absolutely for  $\mathrm{Re}(s)>0$.  We therefore assume that $\mathrm{Re}(s)>0$ until otherwise stated to justify our manipulations.
The integral \eqref{before-q} is equal to $q^{8m}$ times
\begin{align*}
\int_{F^\times}\one_{\OO_F}(t)&\Bigg( \int_{\mathcal{D}_F^{-1} \times \gl_2(F)^{\oplus 2}}
\one_{\varpi^{k+m}\gl_2(\OO_F)^{\oplus 2}}(T-\beta)  \psi\Bigg(\frac{\mathrm{tr}\,\gamma T}{\varpi^{m}t}-\frac{xP(b,T)}{\varpi^{2m}t}\Bigg) dTdx\Bigg)\chi(t)|t|^sdt^\times.
\end{align*} 
If we multiply the integral over $\mathcal{D}_F^{-1} \times \gl_2(F)^{\oplus 2}$ by $\bar{\chi}(t)$ the resulting function of $t$ is invariant under $t \mapsto ut$ for $u \in \OO_F^\times(\varpi^{k+m})$.  Therefore the integral vanishes if the absolute norm of the conductor of $\chi$ is sufficiently large in a sense depending only on $k,m$.  

Letting $\ell\geq 0$ be the smallest integer such that $b_1\varpi^\ell$ and $b_2 \varpi^\ell$ are both integral we see that the above is 
\begin{align*}
\int_{F^\times}\one_{\OO_F}(t)&\Bigg( \int_{\mathcal{D}_F^{-1} \times \gl_2(F)^{\oplus 2}}
\one_{\varpi^{k+m}\gl_2(\OO_F)^{\oplus 2}}(T-\beta)  \psi\Bigg(\frac{\mathrm{tr}\,\gamma T}{\varpi^{m}t}-\frac{ x P(\varpi^\ell b,T)}{\varpi^{2m+\ell}t}\Bigg)dTdx\Bigg)\chi(t)|t|^sdt^\times. 
\end{align*}
This is zero unless $\gamma_1,\gamma_2 \in \varpi^{-(m+\ell+k)}\mathcal{D}_F^{-1}\gl_2(\OO_F)$.

With the vanishing statements claimed in the proposition proven, we are left with establishing a bound for the integral \eqref{int-cons}.
We now assume that $(\gamma_1,\gamma_2) \neq (0,0)$ and $v(t)>1$.
Taking a change of variables $T \mapsto \varpi^{k+m}T+\beta$ we see that \eqref{int-cons}
is equal to $q^{-8(k+m)}$ times
\begin{align} \label{int-t}
& \int_{\mathcal{D}_F^{-1} \times \gl_2(F)^{\oplus 2}}
\one_{\gl_2(\OO_F)^{\oplus 2}}(T_1,T_2)  \psi\Bigg(\frac{\varpi^{m+\ell}\mathrm{tr} \, \gamma(\varpi^{k+m}T+\beta)-xP(\varpi^\ell b,\varpi^{k+m}T+\beta)}{\varpi^{2m+\ell}t}\Bigg)dTdx. \nonumber
\end{align} 
We apply the the $p$-adic stationary phase method of Dabrowski and Fisher \cite{DF} to estimate this integral.  
Choose a generator $\delta \in \OO_F$ for the ideal $\mathcal{D}_F$  (this $\delta$ has no relation to the $\delta^S$-function from \S \ref{sec-delta}).
For $\iota=1, 2$, let
\begin{align*}
\mathcal{F}_{\iota,x}(T_\iota):=&\delta\Big(x(-1)^{\iota-1}b_\iota\varpi^\ell\det (\varpi^{k+m}T_\iota+\beta_\iota)+\varpi^{\ell+m}\mathrm{tr}(\gamma_\iota(\varpi^{k+m}T_\iota+\beta_\iota))\Big)
\end{align*}
and let $\mathcal{F}_x(T_1,T_2)=\mathcal{F}_{1,x}(T_1)+\mathcal{F}_{2,x}(T_2)$.
Let $p$ be the rational prime below $v$ and let 
$$
D_x:=\mathrm{Res}_{\OO_F/\ZZ_p}\mathrm{Spec}(\OO_F[T_1,T_2]/(\nabla \mathcal{F}_x)) \subseteq
\A_{\ZZ_p}^{8[F:\QQ_p]} \quad \textrm{ (affine }8[F:\QQ_p]\textrm{ space).}
$$
Here $(\nabla \mathcal{F}_x) \leq \OO_F[T_1,T_2]$ is the ideal generated by the entries of 
$\nabla \mathcal{F}_x$.  
Writing $T_\iota=(t_{\iota,ij})$, $\gamma_\iota=(\gamma_{\iota,ij})$, and $\beta_\iota=(\beta_{\iota,ij})$ one has
\begin{align*}
\nabla \mathcal{F}_{\iota,x}(T_\iota)=\delta\varpi^{k+\ell+m}\begin{pmatrix} \varpi^{m}\gamma_{\iota,11}-(-1)^{\iota-1}xb_\iota(\varpi^{k+m}t_{\iota,22}+\beta_{\iota,22}) \\
\varpi^{m}\gamma_{\iota,12}+(-1)^{\iota-1}xb_\iota(\varpi^{k+m}t_{\iota,12}+
\beta_{\iota,12})\\
\varpi^{m}\gamma_{\iota,21}+(-1)^{\iota-1}xb_\iota(\varpi^{k+m}t_{\iota,21}+
\beta_{\iota,21})\\
\varpi^{m}\gamma_{\iota,22}-(-1)^{\iota-1}xb_\iota(\varpi^{k+m}t_{\iota,11}+\beta_{\iota,11})\end{pmatrix}
\end{align*}
and thus the Hessian is
\begin{align*}
H_x(T_\iota)=(-1)^{\iota-1}\delta xb_\iota\varpi^{2k+\ell+2m}\begin{pmatrix} & & & -1 \\ & & 1 &\\ & 1 & &\\ -1 & & &  \end{pmatrix}
\end{align*}
which has valuation $v(x^4b_\iota^4 \delta^4)+8k+4\ell+8m$.  Thus the valuation of the Hessian of $\nabla \mathcal{F}_x(T_1,T_2)$ is $v(x^8b_1^4b_2^4 \delta^8)+16k+8\ell+16m$.  
Since $(\gamma_1,\gamma_2) \neq (0,0)$, at any point where $\nabla \mathcal{F}_x$ vanishes we must have $x \neq 0$.  It follows that $|D_x(\ZZ_p)|\leq 1$ and
$v(x) \leq \min(v(\gamma_{1,ij}),v(\gamma_{2,i'j'}))+\kappa'$ for some real number $\kappa'$ depending only on $k,\ell,m$.  Applying \cite[Theorem 1.8]{DF}
we see that \eqref{int-t}
is bounded by 
\begin{align}
|t|^8|t|^{-8/2}\left(\max_{x \in\mathcal{D}^{-1}_F,\alpha \in D_x(\ZZ_p)}|H_x(\alpha)|^{-1}\right)^{1/2} \leq |t|^4 q^\kappa q^{4\min(v(\gamma_{1,ij}),v(\gamma_{2,i'j'}))}
\end{align}
for some real number $\kappa$ depending only on $k,\ell,m$.
This completes the proof of the proposition.

\end{proof}

The following is a corollary:

\begin{cor} \label{cor-h}
Let  $h_0(x,y)=\one_{\OO_F^\times}(x)$ or $\one_{\OO_F^{\times}}(x/y)$ and $f \in C_c^\infty(\GL_2(F)^{\times 2})$.  
Assume that $|b_1| \asymp |b_2| \asymp 1$.
For $\mathrm{Re}(s)>0$ the integral 
$$
\int_{F^\times \times \gl_2(F)^{\oplus 2}}\one_{\OO_F}(t)
h_0(t,P(b,T))
f(T) \psi\left(\frac{\mathrm{tr}\,\gamma T}{t}\right)\chi(t)|t|^sdTdt^\times
$$
converges absolutely.  It vanishes if $|\gamma|$ or the conductor of $\chi$ is sufficiently large in a sense depending only on $f$ and the bounds on $b_1,b_2$.  Finally, 
$$
\int_{F^\times}\one_{\OO_F}(t)\Bigg|\int_{\gl_2(F)^{\oplus 2}}
h(t,P(b,T))
f(T) \psi\left(\frac{\mathrm{tr}\,\gamma T}{t}\right)dT\Bigg||t|^sdt^\times
$$
 is bounded in the half-plane $\mathrm{Re}(s) >-4$ by $q^{4\mathrm{min}(v(\gamma_{1,ij}),v(\gamma_{2,i'j'}))}$ times a constant depending only on $f$.

\end{cor}

\begin{proof}
The assertions when $h_0(x,y)$ is $\one_{\OO_F^\times}(x)$ are clear, for $\one_{\OO_F^\times}(t)f(T)$ is smooth and compactly supported as a function of $(t,T) \in F^\times \times \gl_2(F)^{\oplus 2}$.

If $h_0(x,y)=\one_{\OO_F^{\times}}(y/x)$ then
 the fact that $f \in C_c^\infty(\GL_2(F)^{\times 2})$ 
implies that 
\begin{align} \label{single-t}
\int_{\gl_2(F)^{\oplus 2}}
\one_{\OO_F^\times}\left(\frac{P(b,T)}{t}\right)
f(T) \psi\left(\frac{\mathrm{tr}\,\gamma T}{t}\right)dT
\end{align}
vanishes if  $|t|$ is sufficiently large in a sense depending on $f$ and the bounds on $|b_1|,|b_2|$.  Thus there is an $\ell \in \ZZ$ depending on $f$ and the bounds on $b_1,b_2$ such that 
\begin{align*}
&\int_{F^\times} \Bigg(\int_{\gl_2(F)^{\oplus 2}}
\one_{\OO_F^\times}\left(\frac{P(b,T)}{t}\right)
f(T) \psi\left(\frac{\mathrm{tr}\,\gamma T}{t}\right)dT\Bigg)\chi(t)|t|^sdt^\times\\
&=\int_{F^\times} \Bigg(\int_{\gl_2(F)^{\oplus 2}} \one_{\OO_F}(t)
\one_{t\OO_F}\left(\varpi^{-\ell}P(b,T)\right)
f(T) \psi\left(\frac{\mathrm{tr}(\varpi^{-\ell} \gamma T)}{t}\right)dT\Bigg)\chi(\varpi^{\ell}t)|\varpi^{\ell}t|^sdt^\times.
\end{align*}
We can now apply Proposition \ref{prop-na-bound} and trivial bounds when $0 \leq v(t) \leq 1$ to both summands to deduce the corollary.

\end{proof}

\section{Poisson summation in $d \in F^\times$} \label{sec-ps-in-d}

In this section we prove the following theorem:
\begin{thm} \label{thm-main-comp}
The limit $
\lim_{X \to \infty} \Sigma^0(X)$ exists and is equal to 
\begin{align*}
&\frac{|\det g_\ell g_r^{-1}|^2}{d_F^4}
\sum_{0 \neq \gamma \in \gl_2(F)^{\oplus 2}}\sum_{\substack{b \in (F^\times)^{\oplus 2}\\b_2 \det \gamma_1=b_1 \det \gamma_2}}
\sum_c |c|_S^2\mathcal{I}(b \det g_\ell g_r^{-1},cg_r^{-1}\gamma g_\ell )\one_{\gl_2(\widehat{\OO}_{F}^S)}(g^{-1}_r \gamma g_\ell).
\end{align*}
where the sum on $c$ is over a set of representatives for the principal nonzero ideals of $\OO_F^S$.  The sum over $b,\gamma$ is absolutely convergent.

\end{thm}

Here $\mathcal{I}(b,\gamma)$ is defined as in \eqref{I-def}.  Theorem \ref{thm-main-comp} together with Proposition \ref{prop-spec} and Lemma \ref{lem-vanish} yield Theorem \ref{main-thm}, our main theorem.

\begin{proof}

In \eqref{after-first-ps} we found that $\Sigma^0(X)$ was equal to 
\begin{align*}
d_F^{-4}\sum_{0 \neq \gamma \in \gl_2(F)^{\oplus 2}}&\sum_{b \in (F^\times)^{\oplus 2}}\sum_{d \in \OO_F^S-0}\int_{\gl_2(\A_F)^{\oplus 2} }\frac{V(b\det T/X)}{|b_1\det T_1|_SX}\frac{c_{\sqrt{X}}}{\sqrt{X}}\one_{d \widehat{\OO}_F^S}(P(b,T)) \\& \times 
 h\left(\frac{d}{\Delta(\sqrt{X})},\frac{P(b,T)}{\Delta(X)} \right)\one_{\mathcal{F}}(b\det g_\ell g_r^{-1}) f\one_{\gl_2(\widehat{\OO}_F^S)^{\oplus 2}}(g_\ell^{-1} Tg_r) \psi\left(\frac{\mathrm{tr}\,\gamma T}{d}\right)dT.
\end{align*}
Applying propositions \ref{prop-na-comp}, \ref{prop-arch} and \ref{prop-na-bound}, Corollary  \ref{cor-h} we see that it is permissible to apply Poisson summation in $d \in F^\times$ to this expression, which implies that it is equal to 
\begin{align*}
\frac{ 1}{d_F^{9/2}2 \pi i\mathrm{Res}_{s=1}\zeta^\infty_F(s)}
\sum_{\gamma,b}&\sum_{\chi \in (A F^\times \backslash \A_F^\times)^{\wedge}}\int_{\mathrm{Re}(s)=\sigma}\int_{\A_F^\times}
\one_{\widehat{\OO}_F^{S}}(t)\int_{\gl_2(\A_F)^{\oplus 2}}\frac{V(b\det T/X)}{|b_1\det T_1|_SX} \nonumber \\& \times \frac{c_{\sqrt{X}}}{\sqrt{X}}\one_{t \widehat{\OO}_F^S}(P(b,T)) h\left(\frac{t}{\Delta(\sqrt{X})},\frac{P(b,T)}{\Delta(X)} \right)\one_{\mathcal{F}}(b\det g_\ell g_r^{-1}) \nonumber \\& \times  f\one_{\gl_2(\widehat{\OO}_F^S)^{\oplus 2}}(g^{-1}_\ell Tg_r) \psi\left(\frac{\mathrm{tr}\,\gamma T}{t}\right)dT\chi(t)|t|^{s}dt^\times ds.
\end{align*}
Here we take $\sigma$ sufficiently large (trivial bounds imply that $\sigma>1$ is sufficient).  A possible reference for this application of Poisson summation is \cite[\S 2]{BB}, bearing in mind that our measure differs from theirs by a factor of $\zeta_{F \infty}(1)d_F^{-1/2}$.
Taking a change of variables $(t,T) \mapsto \Delta(\sqrt{X})(t,g_\ell T g_r^{-1})$ we arrive at
\begin{align*} 
\frac{ c_{\sqrt{X}}|\det g_{\ell}g_r^{-1}|^2}{d^{9/2}2 \pi i\mathrm{Res}_{s=1}\zeta^\infty_F(s)}
\sum_{\gamma,b,\chi}&\int_{\mathrm{Re}(s)=\sigma}\int_{\A_F^\times}
\one_{\widehat{\OO}_F^{S}}(t)\int_{\gl_2(\A_F)^{\oplus 2}}\frac{V(b\det g_\ell T g_r^{-1})}{|b_1\det g_{\ell 1}T_1g_{r1}^{-1}|_S} \nonumber \\& \times \one_{t \widehat{\OO}_F^S}(P(b,g_\ell Tg_r^{-1})) h\left(t,P(b \det g_\ell g_r^{-1},T) \right)\one_{\mathcal{F}}(b\det g_\ell g_r^{-1}) \nonumber \\& \times  f\one_{\gl_2(\widehat{\OO}_F^S)^{\oplus 2}}( T) \psi\left(\frac{\mathrm{tr}\,g_r^{-1}\gamma g_\ell T}{t}\right)dTX^{(3+s)/2}\chi(t)|t|^{s}dt^\times ds.
\end{align*}
Applying Proposition \ref{prop-na-comp} we can write this as 
\begin{align}\label{before-move}
&\sum_{c}|c|_S^{-s-1}\frac{ c_{\sqrt{X}}|\det g_{\ell}g_r^{-1}|^2}{d_F^{9/2}2 \pi i\mathrm{Res}_{s=1}\zeta^\infty_F(s)}
\sum_{\gamma,b,\chi}\int_{\mathrm{Re}(s)=\sigma}\int_{F_S^\times}\int_{\gl_2(F_S)^{\oplus 2}}\frac{V(b\det g_\ell T g_r^{-1})}{|b_1\det g_{\ell 1}T_1g_{r1}^{-1}|_S}  \\& \times  h\left(t,P(b \det g_\ell g_r^{-1},T) \right)\one_{\mathcal{F}}(b\det g_\ell g_r^{-1})   f( T) \psi\left(\frac{\mathrm{tr}\,g_r^{-1} c\gamma g_\ell T}{t}\right)dTX^{(3+s)/2}\chi_S(t)|t|_S^{s}dt_S^\times \nonumber \\& \times L(s+5,\chi^S)^{-1}
\one_{\gl_2^{\oplus 2}(\OO_F^S)}(g_r^{-1} \gamma g_\ell)\int_{\widehat{\OO}^S_F}
\one_{t\OO_F}(P(b^{-1},\gamma))\chi^S(t)(|t|^{S})^{s+4}(dt^{S})^{\times}. \nonumber
\end{align}
We note that by definition of $\one_{\mathcal{F}}$ and $V$ in \S \ref{ssec-formula} 
and \S \ref{sec-delta-app}, respectively, the sums over $b_1$ and $b_2$ in this expression can be taken to run over a finite set independent of $\chi$, $\gamma$, $T$ and $t$. 

Consider the Dirichlet series 
\begin{align} \label{ds}
D_{\gamma,b,\chi}(s):=\int_{\widehat{\OO}^S_F}
\one_{t\OO_F}(P(b^{-1},\gamma))\chi^S(t)(|t|^{S})^{ s+4}(dt^{S})^{\times}.
\end{align}
  If  $b_1\det \gamma_2  \neq b_2 \det \gamma_1 $, then  \eqref{ds} converges absolutely in the entire complex plane, and it is bounded by
$O_{\varepsilon}(\max(|P(b^{-1},\gamma)|_S^{-\mathrm{Re}(s)-4+\varepsilon},|P(b^{-1},\gamma)|_S^{\varepsilon}))$ for any $\varepsilon>0$.
  If $b_1\det \gamma_2=b_2\det \gamma_1$ and $\gamma \neq (0,0)$ then 
$D_{\gamma,b,\chi}(s)=L(s+4,\chi^S)$.  

Moving all the contours in \eqref{before-move} to the line $\mathrm{Re}(s)=-\tfrac{7}{2}=-\tfrac{1}{2}-3$, we see therefore see that it is equal to the sum of the contribution of the residues at $s=-3$:
\begin{align} \label{after-move-res}
&\frac{|\det g_\ell g_r^{-1}|^2}{d^{9/2}_F\mathrm{Res}_{s=1}\zeta_F^\infty(s)\zeta^S_F(2)}
\sum_{\gamma}\sum_{\substack{b \in (F^\times)^{\oplus 2}\\ b_2\det \gamma_1=b_1 \det \gamma_2}}
\mathrm{Res}_{s=-3}\zeta_F^S(s+4) c_{\sqrt{X}}\mathcal{I}(b \det g_\ell g_r^{-1},g_r^{-1} \gamma g_\ell)  \one_{\gl_2^{\oplus 2}}(g_r^{-1}\gamma g_\ell)
\end{align}
plus
\begin{align} \label{after-move-nres}
&\frac{ |\det g_\ell g_r^{-1}|^2}{d^{9/2}2 \pi i\mathrm{Res}_{s=1}\zeta^\infty_F(s)}\sum_{\gamma,b,\chi}\int_{\mathrm{Re}(s)=-\tfrac{7}{2}}\frac{\one_{\gl_2^{\oplus 2}}(g_r^{-1}\gamma g_\ell)D_{\gamma,b,\chi}(s)}{L(s+5,\chi^S)}\\& \times \nonumber \Bigg(\int_{\gl_2(F_{S_0})^{\oplus 2}}\frac{V(b\det T)}{|b_1\det g_{\ell 1}T_1g_{r1}^{-1}|_S} c_{\sqrt{X}}h\left(t,P(b,T) \right)\one_{\mathcal{F}}(b\det g_\ell g_r^{-1}) \nonumber \\& \times  f(T) \psi\left(\frac{\mathrm{tr}\,g_r^{-1}\gamma g_\ell T}{t}\right)dT\Bigg)
\chi(t)X^{(3+s)/2}|t|^sdt^\times ds. \nonumber
\end{align}
By Proposition \ref{prop-delta} $c_{\sqrt{X}}=\sqrt{d}_F+O_N(\sqrt{X}^{-N})$ for any $N>0$, so to complete the proof it suffices to show that \eqref{after-move-nres} is $O(X^{-1/4})$.  This is an easy consequence of Proposition \ref{prop-arch} and Proposition \ref{prop-na-bound} and the fact that $L(s+4,\chi^{S})$ is  bounded by $C(\chi,\mathrm{Im}(s+4))^{\beta}$ on the line $\mathrm{Re}(s)=-\tfrac{7}{2}$ by standard preconvex bounds \cite[(10)]{Brumley}.

Regarding the statement in the theorem on the absolute convergence of the sum over $\gamma,b$, the same argument as that given below equation \eqref{before-move} implies that 
the sum over $b$ is actually finite.  The absolute convergence of the $\gamma$ sum therefore follows from Proposition \ref{prop-arch} and Corollary \ref{cor-h}.

\end{proof}


\end{document}